\documentclass[twoside, 12pt]{article}
\usepackage[utf8]{inputenc}
\usepackage[T1]{fontenc}
\usepackage{amssymb,amsmath,amsthm,mathtools,mathrsfs}
\usepackage{enumitem}
\usepackage[colorlinks,bookmarks,linkcolor=black,citecolor=black]{hyperref}
\usepackage{graphicx}
\usepackage{color}
\usepackage[top=2cm, bottom=2cm, left=2cm, right=2cm]{geometry}
\usepackage{float}

\newcommand{\bd}{\begin{description}}
\newcommand{\ed}{\end{description}}
\newcommand{\bi}{\begin{itemize}}
\newcommand{\ei}{\end{itemize}}
\newcommand{\be}{\begin{enumerate}}
\newcommand{\ee}{\end{enumerate}}
\newcommand{\beq}{\begin{equation}}
\newcommand{\eeq}{\end{equation}}
\newcommand{\beqs}{\begin{eqnarray*}}
\newcommand{\eeqs}{\end{eqnarray*}}

\definecolor{DarkGreen}{rgb}{0.2, 0.6, 0.3}

\newtheorem{theorem}{Theorem}

\newtheorem{observation}{Observation}
\newtheorem{lemma}{Lemma}
\newtheorem{definition}{Definition}

\newtheorem{question}{Question}

\def\endofClaim{\hfill\scalebox{.6}{$\blacksquare$}}
\newcommand{\oldqed}{}

\begin{document}
\title{The size-Ramsey numbers of $k$-uniform relaxed $\ell$-trees}

\author{
Meng Ji\footnote{School of Mathematical Sciences, and Institute of Mathematics and Interdisciplinary Sciences, Tianjin Normal University, Tianjin, China. {\tt
mji@tjnu.edu.cn}}
\and
Yulai Ma\footnote{Center for Combinatorics and LPMC, Nankai University, Tianjin 300071, P.R.China.  {\tt
ylma@nankai.edu.cn}}
}
\maketitle

\begin{abstract}
The size-Ramsey number $\widehat{R}_k(\mathcal G)$ of a $k$-uniform hypergraph $\mathcal G$ is the minimum number of edges in a $k$-uniform hypergraph $\mathcal H$ such that every $2$-edge-coloring of $\mathcal H$ contains a monochromatic copy of $\mathcal G$. Fox asked whether, for fixed $2\le \ell<k$, the size-Ramsey number of every $k$-uniform relaxed $\ell$-tree on $n$ vertices is bounded by a polynomial in $n$, as recorded by Dudek, La Fleur, Mubayi and R\"odl.

We give a negative answer when $\ell+1\le k\le 2\ell-2$. More precisely, for every sufficiently large $n$, we construct a $k$-uniform relaxed $\ell$-tree $\bar{\mathcal T}_{n,\ell}^{(k)}$ on exactly $n$ vertices such that
\[
\widehat{R}_k\bigl(\bar{\mathcal T}_{n,\ell}^{(k)}\bigr)
\ge 2^{c_{k,\ell}n^{\,1-1/(2\ell-k)}},
\]
where $c_{k,\ell}>0$ depends only on $k$ and $\ell$. Moreover, any two
distinct edges of the construction intersect in at most $2\ell-k-1$
vertices.  In contrast, for $k\ge 2\ell$, we prove that every $k$-uniform relaxed $\ell$-tree $\bar{\mathcal T}_{n,\ell}^{(k)}$ on $n$ vertices satisfies
\[
\widehat{R}_k\bigl(\bar{\mathcal T}_{n,\ell}^{(k)}\bigr)
=O_{k,\ell}(n^{2k}).
\]
\end{abstract}
\section{Introduction}
The size-Ramsey number of a hypergraph measures the number of edges needed
to force a monochromatic copy under every two-coloring. More precisely, a
$k$-uniform hypergraph, or \emph{$k$-graph}, $\mathcal G$ consists of a vertex set
$V(\mathcal G)$ and an edge set $E(\mathcal G)\subseteq\binom{V(\mathcal G)}k$. For $k$-graphs $\mathcal G$ and
$H$, write $\mathcal H\longrightarrow_k \mathcal G$ if every red-blue edge-coloring of $H$
contains a monochromatic copy of $\mathcal G$. The Ramsey number and the
size-Ramsey number of $\mathcal G$ are, respectively,
\[
R_k(\mathcal G):=\min\{N:K_N^{(k)}\longrightarrow_k\mathcal G\},
\qquad
\widehat{R}_k(\mathcal G):=\min\{|E(\mathcal H)|:\mathcal H\longrightarrow_k\mathcal G\}.
\]
where $K_N^{(k)}$ is the complete $k$-graph on $N$ vertices. Throughout,
copies are not required to be induced.

The study of size-Ramsey numbers began with graphs, in work of Erd\H{o}s,
Faudree, Rousseau, and Schelp~\cite{erdHos1978size}.
Subsequent work has obtained bounds for paths~\cite{beck1983size} and
bounded-degree trees~\cite{friedman1987expanding}, as well as more general
bounded-degree graphs \cite{draganic2025size,rodl2000size}. For bounded-degree graphs, bounds have also been obtained in terms
of treewidth, both when it is fixed ~\cite{kamcev2021size} and when it can grow with the number of vertices \cite{draganic2026size}. Further results concern fixed powers of bounded-degree
trees ~\cite{berger2021size}, short
subdivisions of bounded-degree graphs~\cite{draganic2021size,kohayakawa2019size},
and grids~\cite{conlon2023size}.

Dudek, La Fleur, Mubayi, and R\"odl~\cite{dudek2017size}
initiated the corresponding study for uniform hypergraphs, considering
cliques, paths, trees, and bounded-degree hypergraphs. Building on this, the bounds for tight paths were 
established~\cite{han2021size,letzter2025size,lu2018size}. 
Letzter, Pokrovskiy, and Yepremyan~\cite{letzter2021size}
studied fixed powers of bounded-degree hypergraph trees. For related
results on powers of graph paths and multicolor lower
bounds for hypergraph paths, see~\cite{bal2024lower,clemens2019size}.

In this paper, we investigate how the
size-Ramsey behavior of hypergraph trees depends on the way successive
edges attach to the preceding edges.
We use the following two notions of tree from~\cite{dudek2017size}.

\begin{definition}
Let $1\le\ell<k$. An $n$-vertex $k$-graph $\mathcal T$ is an \emph{$\ell$-tree},
denoted $\mathcal{T}_{n,\ell}^{(k)}$, if it admits an ordering $e_1,\ldots,e_m$ of
its edges such that, for every $2\le j\le m$, the set
\[
 S_j=e_j\cap\bigcup_{i<j}e_i
\]
satisfies $|S_j|\le\ell$ and $S_j\subseteq e_{i_0}$ for some $i_0<j$.
It is a \emph{relaxed $\ell$-tree}, denoted $\bar{\mathcal{T}}_{n,\ell}^{(k)}$, if it
admits an edge ordering satisfying only $|S_j|\le\ell$. Such an ordering
is called a \emph{witnessing ordering}.
\end{definition}

Thus every new edge of a relaxed $\ell$-tree intersects the union of the earlier
edges in at most $\ell$ vertices. In an $\ell$-tree, these previously used
vertices must additionally lie together in a single earlier edge. We
sometimes call the former a \emph{strict $\ell$-tree} to distinguish the
two notions. The definitions permit disconnected hypergraphs and
isolated vertices. When $\ell=1$, the two notions coincide; for
$\ell\ge2$, the containment condition imposes a genuine additional
restriction.

Dudek, La Fleur, Mubayi, and R\"odl proved the following polynomial upper bound for strict $\ell$-trees.

\begin{theorem}[\cite{dudek2017size}]
Let $1\le\ell<k$ be fixed integers. Every $n$-vertex $k$-uniform
$\ell$-tree satisfies
$$\widehat{R}_k\left(\mathcal{T}_{n,\ell}^{(k)}\right) = O_{k,\ell}(n^{\ell+1}).$$
\end{theorem}

It is natural to ask whether a polynomial bound survives after the
containment condition is removed. Dudek, La Fleur, Mubayi, and R\"odl~\cite[Question~2.6]{dudek2017size} record the following open question pointed out by Fox (personal communication, 2014).

\begin{question}[Dudek, La Fleur, Mubayi and R\"{o}dl~\cite{dudek2017size}; Fox]
Let $2 \le \ell < k$ be fixed integers. Is $\widehat{R}_{k}(\bar{\mathcal{T}}_{n,\ell}^{(k)})$ bounded by a polynomial in $n$?
\end{question}

In this paper, we give a negative answer in the range $\ell+1\leq k\leq 2\ell-2$ and a positive answer in the range $k\geq 2\ell$.  

\begin{theorem}\label{main-thm}
For every integer $\ell\ge3$ and every integer $k$ satisfying
$\ell+1\le k\le 2\ell-2$, 
there is a constant $c_{k,\ell}>0$ such that, for every sufficiently large
integer $n$, there is a relaxed $\ell$-tree $\bar{\mathcal{T}}_{n,\ell}^{(k)}$ on exactly $n$
vertices with
\[
   \widehat{R}_k(\bar{\mathcal{T}}_{n,\ell}^{(k)})\ge 2^{c_{k,\ell}n^{\,1-1/(2\ell-k)}}.
\]
Moreover, any two distinct edges of $\bar{\mathcal{T}}_{n,\ell}^{(k)}$ intersect in at most
$2\ell-k-1$ vertices.
\end{theorem}

The intersection bound shows that the obstruction persists even when
pairwise edge intersections are small. In particular, by taking
$\ell=3$ and $k=4$, we obtain for every sufficiently large $n$ a
\emph{linear} $4$-uniform relaxed $3$-tree with size-Ramsey number at
least $2^{c\sqrt n}$, where linear means that any two distinct edges
intersect in at most one vertex. Thus linearity alone does not restore
the polynomial bound enjoyed by strict trees.

For $k\ge2\ell$, we obtain a polynomial size-Ramsey bound by proving
the stronger statement that Ramsey number is at most quadratic
in the number of vertices.

\begin{theorem}\label{thm-polynomial}
Let $2\leq\ell$ and $k\ge2\ell$ be integers, and let $\bar{\mathcal{T}}_{n,\ell}^{(k)}$ be an $n$-vertex $k$-uniform relaxed $\ell$-tree with $m\ge1$ edges. Then
\begin{equation*}\label{eq:poly-main}
R_k(\bar{\mathcal{T}}_{n,\ell}^{(k)})\le2\ell(\ell+1)n(2m-1)+2kn. 
\end{equation*}
In particular, for fixed $k$ and $\ell$,
\begin{equation*}\label{eq:poly-asymptotic}
 R_k(\bar{\mathcal{T}}_{n,\ell}^{(k)})=O_{k,\ell}(n^2),\qquad
 \widehat{R}_k(\bar{\mathcal{T}}_{n,\ell}^{(k)})=O_{k,\ell}(n^{2k}).
\end{equation*}
\end{theorem}

We briefly describe the two arguments.  For the lower bound, inspired by Kostochka and R{\"o}dl~\cite{kostochka2006ramsey}, Conlon, Fox, and R{\"o}dl~\cite{conlon2015hedgehogs}, Fox and Li~\cite{fox2020ramsey}, and Dubroff, Gir{\~a}o, Hurley, and Yap~\cite{dubroff2023tower}, we use a sparse version of the generalized
hedgehog studied in~\cite{dubroff2023tower}.
Set $r=2\ell-k\ge2$. We choose a family of
$\Theta_{k,\ell}(t^r)$ body $\ell$-sets on at most $t$
vertices, with pairwise intersections of size at most
$r-1$, and attach a private $(k-\ell)$-set to each body.
The packing condition makes the events that distinct
bodies span a red $K_\ell^{(r)}$ independent under a
random coloring of the auxiliary $r$-sets, and likewise
for blue. A union bound yields an auxiliary coloring on
$2^{\Omega_{k,\ell}(t^{r-1})}$ vertices such that, under
every injection of the body vertices, some selected body
spans a red $K_\ell^{(r)}$ and another spans a blue
$K_\ell^{(r)}$.
We color a $k$-set red if it contains a red
$K_\ell^{(r)}$, and blue otherwise.
Since any two $\ell$-subsets of a $k$-set share an
$r$-set, a $k$-set containing a blue $K_\ell^{(r)}$
cannot also contain a red one.
Thus every copy of the construction has edges of both
colors, giving the desired Ramsey lower bound.
The construction has $\Theta_{k,\ell}(t^r)$ vertices,
so converting this into a size-Ramsey bound yields
the exponent $(r-1)/r$.

For the upper bound, the main ingredient is the
additive deletion lemma. It provides a color $c$ and a set of
$O_\ell(n)$ vertices whose deletion leaves every remaining
$\ell$-set with color-$c$ extensions in the original vertex set
having transversal number at least $n$.
Repeated applications produce nested vertex sets, with sufficiently
many stages carrying the same color.
Working backwards through these stages, we embed the target edges
in a witnessing ordering. At each step, the previously embedded
vertices lie in the current vertex set, and the transversal bound
provides an extension avoiding all used vertices outside the
current edge. This produces a monochromatic embedding in a host
of order $O_{k,\ell}(n^2)$, and hence the polynomial
size-Ramsey bound.

The lower and upper bounds are proved in Sections~2 and~3,
respectively. Section~4 discusses the remaining boundary case
$k=2\ell-1$.

\section{Proof of Theorem \ref{main-thm}}
Before proving Theorem~\ref{main-thm}, we establish a packing and
coloring lemma. 

\begin{lemma}\label{lem:coloring}
For fixed integers $\ell\ge3$ and $2\le r\le\ell$,
there is a constant $a_{\ell,r}>0$ such that the following
holds for every sufficiently large integer $t$.
Let $\mathcal F\subseteq\binom{[t]}{\ell}$ be
inclusion-maximal subject to
\[
|A\cap B|\le r-1
\qquad\text{for distinct }A,B\in\mathcal F,
\]
and set $W=\bigcup_{A\in\mathcal F}A$ and
$b=\binom{\ell}{r}$. Then
\[
\frac{\binom{t}{r}}{b^2}
\le |\mathcal F|
\le \frac{\binom{t}{r}}{b}.
\]
Moreover, the edges of the complete $r$-graph $K_N^{(r)}$,
where $N=\lfloor2^{a_{\ell,r}t^{r-1}}\rfloor$,
can be colored red and blue so that, for every injection
$\varphi:W\to[N]$, there are $A_R,A_B\in\mathcal F$
such that $\varphi(A_R)$ spans a red $K_\ell^{(r)}$
and $\varphi(A_B)$ spans a blue $K_\ell^{(r)}$.
\end{lemma}

\begin{proof}
Set $p=|\mathcal F|$. Each $r$-set is contained in at most
one member of $\mathcal F$, so $pb\le\binom{t}{r}$.
By maximality, every $\ell$-set shares an $r$-set with
some member of $\mathcal F$. We count the \(\ell\)-sets containing one of the \(pb\) selected \(r\)-sets, which gives
\[
\binom{t}{\ell}
\le pb\binom{t-r}{\ell-r}.
\]
Since
\[
\frac{\binom{t}{\ell}}{\binom{t-r}{\ell-r}}
=\frac{\binom{t}{r}}{\binom{\ell}{r}},
\]
the lower bound follows. In particular, for $t\ge2r$,
\[
p\ge\gamma_{\ell,r}t^r,
\qquad
\gamma_{\ell,r}:=\frac{1}{2^r r!b^2}.
\]

Color the edges of $K_N^{(r)}$ independently and uniformly
red or blue. Fix an injection $\varphi:W\to[N]$.
For each $A\in\mathcal F$, the probability that
$\varphi(A)$ spans a red $K_\ell^{(r)}$ is
$\rho:=2^{-b}$. These events are independent over
$A\in\mathcal F$, since distinct members of $\mathcal F$
share no $r$-set. Thus
\[
\Pr\bigl(
\text{no }\varphi(A),\ A\in\mathcal F,
\text{ spans a red }K_\ell^{(r)}
\bigr)
=(1-\rho)^p
\le e^{-\rho p}.
\]
The same bound holds for the absence of a blue
$K_\ell^{(r)}$ among these images.

There are at most $N^{|W|}\le N^t$ injections.
Choose
\[
a_{\ell,r}:=\frac{\rho\gamma_{\ell,r}}{4\ln2},
\qquad
N:=\left\lfloor2^{a_{\ell,r}t^{r-1}}\right\rfloor.
\]
By the union bound, the probability that some injection
fails to have both a red and a blue $K_\ell^{(r)}$
among the images of members of $\mathcal F$ is at most
\begin{equation*}
2N^t e^{-\rho p}\le2e^{(a_{\ell,r}\ln2-\rho\gamma_{\ell,r})t^r}=2e^{-3\rho\gamma_{\ell,r}t^r/4}<1 
\end{equation*}
for all sufficiently large $t$.
This proves the lemma.
\end{proof}

\begin{proof}[Proof of Theorem~\ref{main-thm}]
Fix $\ell+1\le k\le2\ell-2$, and set $d:=k-\ell,
r:=2\ell-k\ge2$.
Let $\mathcal F\subseteq\binom{[t]}{\ell}$ and
$W=\bigcup_{A\in\mathcal F}A$ be as in
Lemma~\ref{lem:coloring}. For each $A\in\mathcal F$,
choose a private set $Z_A$ of size $d$, where all the $Z_A$
are pairwise disjoint and disjoint from $W$.
Define the $k$-graph $\mathcal H_t$ by
\[
V(\mathcal H_t)
=
W\cup\bigcup_{A\in\mathcal F}Z_A,
\qquad
E(\mathcal H_t)
=
\{A\cup Z_A:A\in\mathcal F\}.
\]
Writing \(p=|\mathcal F|\), we obtain from Lemma~\ref{lem:coloring} that
\[
q_t:=|V(\mathcal H_t)|
=|W|+dp
=\Theta_{k,\ell}(t^r).
\]

In any ordering of these edges, the private vertices
of the current edge have not occurred in an earlier edge.
Its intersection with the preceding vertex union is
therefore contained in its body and has size at most $\ell$.
Thus $\mathcal H_t$ is a relaxed $\ell$-tree.
It has no isolated vertices, since $W$ is the union of
the members of $\mathcal F$. Moreover, for distinct
$A,B\in\mathcal F$,
\[
|(A\cup Z_A)\cap(B\cup Z_B)|
=|A\cap B|
\le r-1.
\]

Take $t$ sufficiently large for
Lemma~\ref{lem:coloring} to apply, and let
$N=\left\lfloor2^{a_{\ell,r}t^{r-1}}\right\rfloor$. 
Use the red-blue coloring of $K_N^{(r)}$ supplied by
Lemma~\ref{lem:coloring} to color the edges of $K_N^{(k)}$:
a $k$-set $Q$ is red if it contains a red
$K_\ell^{(r)}$, and blue otherwise.
If two $\ell$-sets $A,B$ lie in the same $k$-set, then
\[
|A\cap B|\ge2\ell-k=r.
\]
They therefore share an $r$-set, so they cannot span
a red and a blue $K_\ell^{(r)}$, respectively.
Consequently,
\begin{equation*}
Q\text{ contains a red }K_\ell^{(r)}
\Longrightarrow Q\text{ is red}, \quad
Q\text{ contains a blue }K_\ell^{(r)}
\Longrightarrow Q\text{ is blue}.
\end{equation*}

Suppose, for a contradiction, that this coloring contains
a monochromatic copy of $\mathcal H_t$, with embedding
$\varphi$. The restriction $\varphi|_W$ is an injection
from $W$ into $[N]$. By Lemma~\ref{lem:coloring},
there are $A_R,A_B\in\mathcal F$ such that
$\varphi(A_R)$ spans a red $K_\ell^{(r)}$
and $\varphi(A_B)$ spans a blue $K_\ell^{(r)}$.
It follows that
$\varphi(A_R\cup Z_{A_R})$ is red and
$\varphi(A_B\cup Z_{A_B})$ is blue, a contradiction.
Hence
\[
R_k(\mathcal H_t)>N.
\]

We use the elementary inequality
\begin{equation}\label{eq:Ramsey-size}
R_k(\mathcal G)\le k\widehat R_k(\mathcal G)
\end{equation}
for every $k$-graph $\mathcal G$ without isolated vertices.
Indeed, suppose $\mathcal F\longrightarrow_k\mathcal G$
and $|E(\mathcal F)|=M$. Deleting the isolated vertices
of $\mathcal F$ leaves a $k$-graph $\mathcal F'$
on at most $kM$ vertices, and
$\mathcal F'\longrightarrow_k\mathcal G$
because $\mathcal G$ has no isolated vertices.
Since $\mathcal F'$ embeds into $K_{kM}^{(k)}$,
we have $R_k(\mathcal G)\le kM$.
Minimizing over $\mathcal F$ proves
\eqref{eq:Ramsey-size}.

Applying \eqref{eq:Ramsey-size} to $\mathcal H_t$,
we obtain, for all sufficiently large $t$,
\begin{equation*}
\widehat R_k(\mathcal H_t)>\frac{N}{k}\ge2^{a_{\ell,r}t^{r-1}-1-\log_2 k}\ge2^{(a_{\ell,r}/2)t^{r-1}}. 
\end{equation*}

It remains to obtain every sufficiently large order
while preserving the intersection bound.
Given $n$, choose
\[
t:=
\left\lfloor
\left(\frac{n}{2(d+1)}\right)^{1/r}
\right\rfloor.
\]
For sufficiently large $n$, all the preceding estimates
apply, and $t=\Theta_{k,\ell}(n^{1/r})$.
Since $|W|\le t$ and $p\le t^r$, we have
\[
q_t\le t+dt^r
\le(d+1)t^r
\le n/2.
\]
Set $q:=n-q_t$. For sufficiently large $n$, we can write
\[
q=u(k-1)+vk,
\qquad u,v\ge0.
\]
Indeed, choose $v\in\{0,\ldots,k-2\}$ with
$v\equiv q\pmod{k-1}$ and set
$u=(q-vk)/(k-1)$.
Then $u$ is a nonnegative integer whenever
$q\ge k(k-2)$.

Fix $x\in V(\mathcal H_t)$.
Add $u$ edges, each consisting of $x$ and $k-1$
fresh vertices, and add $v$ disjoint fresh $k$-edges.
All fresh vertex sets are pairwise disjoint and disjoint
from $V(\mathcal H_t)$.
The resulting $k$-graph $\mathcal T_n$ has exactly
$n$ vertices. Appending the new edges to a witnessing
ordering of $\mathcal H_t$ shows that $\mathcal T_n$
is a relaxed $\ell$-tree.
Each new edge intersects any other edge in at most
one vertex. Since $r\ge2$, any two distinct edges
of $\mathcal T_n$ therefore intersect in at most $r-1$
vertices.

By monotonicity,
\begin{equation*}
\widehat R_k(\mathcal T_n)
\ge\widehat R_k(\mathcal H_t)
\ge2^{(a_{\ell,r}/2)t^{r-1}}
\ge2^{c_{k,\ell}n^{(r-1)/r}}
\end{equation*}
for some constant $c_{k,\ell}>0$ and all sufficiently
large $n$. Since
\[
\frac{r-1}{r}=1-\frac{1}{2\ell-k},
\]
this proves the theorem.
\end{proof}

\section{Proof of Theorem \ref{thm-polynomial}}
\label{sec:polynomial}
Before proving Theorem \ref{thm-polynomial},
we begin with a simple observation and an auxiliary lemma.
\begin{observation}\label{ob-elementary}
For $m\geq1$, let $T$ be an $n$-vertex, $m$-edge $k$-uniform relaxed $\ell$-tree and let $d=k-\ell$. Then $n\ge k+d(m-1)=\ell+dm$.

\end{observation}


Indeed, the first edge uses $k$ vertices, and every
subsequent edge introduces at least $k-\ell$ new vertices.

Here we introduce some necessary notation.  Fix a red-blue coloring
of $\binom{U}{k}$. For $A\in\binom U\ell$ and $c\in\{\mathrm{R},\mathrm{B}\}$, define
\[
 \mathcal{L}_c^U(A)=\{Q\setminus A:A\subseteq Q\in\tbinom Uk,\ Q\text{ has color }c\},
\]
and $Q\setminus A$ is an \textit{extension} of $A$.
A \emph{transversal} of a family is a vertex set
intersecting each member. Write $\tau_c^U(A)$ for the smallest transversal size of $\mathcal{L}_c^U(A)$, with $\tau_c^U(A)=0$ if the family is empty. Since members of the link avoid $A$, its transversals can always be chosen in $U\setminus A$.
Thus, if $\tau_c^U(A)\ge h$ and $F\subseteq U\setminus A$ has $|F|<h$, then some color-$c$ edge $Q$ through $A$ satisfies
\begin{equation}\label{eq:avoid}
 (Q\setminus A)\cap F=\varnothing.
\end{equation}
Otherwise $F$ itself would be a transversal. Call $A$ \emph{$c$-bad at threshold $h$} when $\tau_c^U(A)<h$.


\begin{lemma}\label{lem:additive}
Let $k,\ell,h$ be positive integers with
$k\ge2\ell$, and suppose that $|U|\ge2kh$. For every red-blue coloring of $\binom Uk$, there are a color $c\in\{\mathrm R,\mathrm B\}$ and $X\subseteq U$ satisfying
\begin{equation}\label{eq:additive-size}
 |X|\le2\ell(\ell+1)h
\end{equation}
and
\begin{equation}\label{eq:additive-good}
 \tau_c^U(A)\ge h\qquad\text{for every }A\in\binom{U\setminus X}{\ell}.
\end{equation}
\end{lemma}
\begin{proof}
Let $s=2h$. Suppose, for a contradiction, that both the red-bad and the blue-bad $\ell$-uniform families have matchings of size at least $(\ell+1)s$.
We first choose pairwise disjoint red-bad sets $A_0,\ldots,A_{s-1}$ whose union has $\ell s$ vertices, and hence those intersect at most $\ell s$ vertex-sets of a matching of blue-bad sets. We can therefore choose $s$ blue-bad sets $B_0,\ldots,B_{s-1}$ such that all $2s$ sets $A_i,B_j$ are mutually disjoint.

If $k>2\ell$, we choose pairwise disjoint $(k-2\ell)$-sets $C_0,\ldots,C_{s-1}$ disjoint from all these sets since the total number of vertices is $2\ell s+(k-2\ell)s=ks=2kh\le|U|.$
For $k=2\ell$, take all $C_t=\varnothing$; their nonempty parts are then simply unnecessary. For $0\le i,j<s$, we define
$$
 Q_{ij}=A_i\cup B_j\cup C_{i+j\pmod s}.
$$
Each $Q_{ij}$ is a $k$-set.

For fixed $i$, the sets $Q_{ij}\setminus A_i$ are pairwise disjoint: the $B_j$ are disjoint and the indices $i+j\pmod s$ are distinct. Since $A_i$ is red-bad, there is a set of at most
$h-1$ vertices intersecting every red extension
of $A_i$. In particular, it intersects
$Q_{ij}\setminus A_i$ whenever $Q_{ij}$ is red.
Since these sets are pairwise disjoint as $j$ varies,
at most $h-1$ of the sets $Q_{ij}$ are red.
Similarly, for fixed $j$, the sets $Q_{ij}\setminus B_j$ are pairwise disjoint as $i$ varies. Hence at most $h-1$ of the sets $Q_{ij}$ with $0\le i<s$ are blue.

We now prove that one color $c$ has a bad-set family of matching number less than $(\ell+1)s=2(\ell+1)h$. There are $s$ choices of $i$, and for each fixed $i$ at most $h-1$ of the corresponding $Q_{ij}$ are red; thus the total number of red $Q_{ij}$ is at most $s(h-1)$. Similarly, the total number of blue $Q_{ij}$ is at most $s(h-1)$. Therefore
$s^2\le s(h-1)+s(h-1)=2s(h-1),$
contradicting $s=2h$. Thus one color $c$ has a bad-set family of matching number less than $(\ell+1)s=2(\ell+1)h$.

Let $X$ be the union of a maximal matching in that family. Then $X$ intersects every $c$-bad set, as otherwise the matching could be enlarged. Moreover, $|X|<2\ell(\ell+1)h$, which gives the desired size bound. Thus no $\ell$-set in $U\setminus X$ is $c$-bad, completing the proof.
\end{proof}

\begin{proof}[Proof of Theorem \ref{thm-polynomial}]
Set $D:=2\ell(\ell+1), L:=2m-1, N:=DnL+2kn$.
Fix an arbitrary edge-coloring of $K_N^{(k)}$ with vertex set $U_0=\mathcal{V}(K_N^{(k)})$.
Fix $h=n$ and apply Lemma~\ref{lem:additive} $L$ times. At step $i$, we obtain $X_i\subseteq U_{i-1}$ and a color $c_i$, and set $U_i=U_{i-1}\setminus X_i$. Thus
\begin{equation}\label{eq:poly-nested}
 U_0\supseteq U_1\supseteq\cdots\supseteq U_L,
 \qquad |U_{i-1}\setminus U_i|\le Dn,
\end{equation}
and
\begin{equation}\label{eq:poly-stage}
 \tau_{c_i}^{U_{i-1}}(A)\ge n
 \qquad\text{for every }A\in\binom{U_i}{\ell}.
\end{equation}
Every set needed in these applications has at least $N-LDn=2kn$ vertices, so the size hypothesis of Lemma~\ref{lem:additive} is satisfied at every application.
Clearly, some color $c$ occurs at least $m$ times among the $2m-1$ colors $c_1,\ldots,c_L$. Select such indices in \emph{decreasing} order: $a_1>a_2>\cdots>a_m, c_{a_j}=c$.
Let $e_1,\ldots,e_m$ be a witnessing ordering of $\bar{\mathcal{T}}_{n,\ell}^{(k)}$. We embed $e_j$ using stage $a_j$.

Let $I_j$ be the set of host vertices used before embedding $e_j$.
We maintain the embedding invariant $I_j\subseteq U_{a_j}$.
Initially $I_1=\varnothing$, so the invariant holds vacuously.
Let $S_j=e_j\cap\bigcup_{t<j}e_t,\qquad s_j=|S_j|\le\ell$.
By the embedding invariant, the image of $S_j$ is contained in $U_{a_j}$.
Since $|U_{a_j}|\ge2kn$ and at most $n$ vertices will ever be used in total,
we can extend the image of $S_j$ to a set $A\subseteq U_{a_j}$
of size $\ell$ by adding $\ell-s_j$ unused vertices of $U_{a_j}$,
and we assign these added vertices to $\ell-s_j$ new vertices of $e_j$.

Let $F$ be the set of host vertices already used outside $A$.
Then $F\subseteq U_{a_j-1}\setminus A$, and $|F|<n$, since $e_j$
still has $k-\ell\ge1$ vertices left to embed. By \eqref{eq:poly-stage},
$\tau_{c}^{U_{a_j-1}}(A)\ge n$, so $F$ is not a transversal of the
color-$c$ extensions of $A$. Therefore \eqref{eq:avoid} implies a
color-$c$ edge $Q\subseteq U_{a_j-1}$ with $A\subseteq Q$ and $(Q\setminus A)\cap F=\varnothing.$
We then bijectively map the remaining new vertices of $e_j$ onto $Q\setminus A$.
The embedding remains injective and \(e_j\) receives color \(c\).
Moreover, all vertices embedded so far now lie in \(U_{a_j-1}\).
If \(j<m\), then \(a_j-1\ge a_{j+1}\), so the inclusion $U_{a_j-1}\subseteq U_{a_{j+1}}$
shows that the embedding invariant holds for the next step.
After the $m$ steps, every edge is embedded in color $c$. We assign any isolated vertices of $\bar{\mathcal{T}}_{n,\ell}^{(k)}$ to unused host vertices. Hence $K_N^{(k)}\longrightarrow_k \bar{\mathcal{T}}_{n,\ell}^{(k)}$, proving the claimed Ramsey bound.

By Observation~\ref{ob-elementary}, we have $m\le(n-\ell)/(k-\ell)\le n$. Thus $N=O_{k,\ell}(n^2)$. Since, for any relaxed $\ell$-tree $T$, $\widehat R_k(T)\le\binom{R_k(T)}k\le R_k(T)^k$, we have $\widehat{R}_k(\bar{\mathcal{T}}_{n,\ell}^{(k)})=O_{k,\ell}(n^{2k})$.
\end{proof}

\section{Concluding remarks}

Theorems~\ref{main-thm} and~\ref{thm-polynomial}
leave the boundary case $k=2\ell-1$ unresolved.
For the lower bound, we use a packing of body
$\ell$-sets with pairwise intersections of size
at most $r-1$, where $r=2\ell-k$.
At the boundary, $r=1$, so the bodies must be
pairwise disjoint, and the same random-coloring
argument no longer yields a superpolynomial lower bound.
For the upper bound, the additive-deletion argument
uses two disjoint $\ell$-sets inside a common $k$-set,
which requires $k\ge2\ell$.

In the smallest boundary case $(k,\ell)=(3,2)$, Allen, Boyadzhiyska, and Pavez-Sign\'e~\cite{allen2026ramsey} proved an $O(n^{3/2})$ Ramsey bound
for $n$-vertex $3$-uniform generalized hedgehogs. Their upper bound also yields a polynomial size-Ramsey bound for this subclass of relaxed $2$-trees.

It remains to determine whether, for every fixed
$\ell\ge2$, the size-Ramsey number of every
$(2\ell-1)$-uniform relaxed $\ell$-tree is bounded
by a polynomial in its number of vertices.

\section*{Acknowledgements}
Meng Ji was partially supported by the Tianjin Municipal Education Commission Scientific Research Program Project (No. 2025KJ133). Yulai Ma was partially supported by the National Natural Science Foundation of China (No. 12601673, 12631013).

\subsection*{Declaration of generative AI use}
During the preparation of this work, the authors used ChatGPT to improve the language and readability of the manuscript, 
including grammatical corrections and language polishing. All mathematical results and proofs were developed and verified by the 
authors, who take full responsibility for the content of the article.

\bibliographystyle{plain}
\bibliography{size_ramsey}

@article{erdHos1978size,
  title={The size {R}amsey number},
  author={Erd{\H{o}}s, Paul and Faudree, Ralph J and Rousseau, Cecil C and Schelp, Richard H},
  journal={Periodica Mathematica Hungarica},
  volume={9},
  number={1-2},
  pages={145--161},
  year={1978},
  publisher={Akad{\'e}miai Kiad{\'o}, co-published with Springer Science+ Business Media BV~…}
}

@article{beck1983size,
  title={On size {R}amsey number of paths, trees, and circuits. {I}},
  author={Beck, J{\'o}zsef},
  journal={Journal of Graph Theory},
  volume={7},
  number={1},
  pages={115--129},
  year={1983},
  publisher={Wiley Online Library}
}

@article{friedman1987expanding,
  title={Expanding graphs contain all small trees},
  author={Friedman, Joel and Pippenger, Nicholas},
  journal={Combinatorica},
  volume={7},
  number={1},
  pages={71--76},
  year={1987},
  publisher={Springer}
}

@article{rodl2000size,
  title={On size {R}amsey numbers of graphs with bounded degree},
  author={Rodl, Vojtech and Szemer{\'e}di, Endre},
  journal={Combinatorica},
  volume={20},
  number={2},
  pages={257--262},
  year={2000},
  publisher={Springer-Verlag GmbH}
}

@article{kamcev2021size,
  title={The size {R}amsey number of graphs with bounded treewidth},
  author={Kamcev, Nina and Liebenau, Anita and Wood, David R and Yepremyan, Liana},
  journal={SIAM Journal on Discrete Mathematics},
  volume={35},
  number={1},
  pages={281--293},
  year={2021},
  publisher={SIAM}
}

@article{berger2021size,
  title={The size-{R}amsey number of powers of bounded degree trees},
  author={Berger, S{\"o}ren and Kohayakawa, Yoshiharu and Maesaka, Giulia Satiko and Martins, Ta{\'\i}sa and Mendon{\c{c}}a, Walner and Mota, Guilherme Oliveira and Parczyk, Olaf},
  journal={Journal of the London Mathematical Society},
  volume={103},
  number={4},
  pages={1314--1332},
  year={2021},
  publisher={Wiley Online Library}
}

@article{draganic2021size,
  title={The size-{R}amsey number of short subdivisions},
  author={Dragani{\'c}, Nemanja and Krivelevich, Michael and Nenadov, Rajko},
  journal={Random Structures \& Algorithms},
  volume={59},
  number={1},
  pages={68--78},
  year={2021},
  publisher={Wiley Online Library}
}

@article{conlon2023size,
  title={On the size-{R}amsey number of grids},
  author={Conlon, David and Nenadov, Rajko and Truji{\'c}, Milo{\v{s}}},
  journal={Combinatorics, Probability and Computing},
  volume={32},
  number={6},
  pages={874--880},
  year={2023},
  publisher={Cambridge University Press}
}

@article{dudek2017size,
  title={On the size-{R}amsey number of hypergraphs},
  author={Dudek, Andrzej and Fleur, Steven La and Mubayi, Dhruv and R{\"o}dl, Vojtech},
  journal={Journal of Graph Theory},
  volume={86},
  number={1},
  pages={104--121},
  year={2017},
  publisher={Wiley Online Library}
}

@article{lu2018size,
  title={On the size-{R}amsey number of tight paths},
  author={Lu, Linyuan and Wang, Zhiyu},
  journal={SIAM Journal on Discrete Mathematics},
  volume={32},
  number={3},
  pages={2172--2179},
  year={2018},
  publisher={SIAM}
}

@article{han2021size,
  title={The size-{R}amsey number of 3-uniform tight paths},
  author={Han, Jie and Kohayakawa, Yoshiharu and Letzter, Shoham and Mota, Guilherme Oliveira and Parczyk, Olaf},
  journal={Advances in Combinatorics},
  pages={Paper No.5},
  year={2021}
}

@article{letzter2021size,
  title={Size-{R}amsey numbers of powers of hypergraph trees and long subdivisions},
  author={Letzter, Shoham and Pokrovskiy, Alexey and Yepremyan, Liana},
  journal={arXiv preprint arXiv:2103.01942},
  year={2021}
}

@article{letzter2025size,
  title={Size-{R}amsey numbers of tight paths},
  author={Letzter, Shoham and Pokrovskiy, Alexey and Yepremyan, Liana},
  journal={arXiv preprint arXiv:2507.01498},
  year={2025}
}

@article{clemens2019size,
  title={The size-{R}amsey number of powers of paths},
  author={Clemens, Dennis and Jenssen, Matthew and Kohayakawa, Yoshiharu and Morrison, Natasha and Mota, Guilherme Oliveira and Reding, Damian and Roberts, Barnaby},
  journal={Journal of graph theory},
  volume={91},
  number={3},
  pages={290--299},
  year={2019},
  publisher={Wiley Online Library}
}

@article{bal2024lower,
  title={A lower bound on the multicolor size-{R}amsey numbers of paths in hypergraphs},
  author={Bal, Deepak and DeBiasio, Louis and Lo, Allan},
  journal={European Journal of Combinatorics},
  volume={120},
  pages={103969},
  year={2024},
  publisher={Elsevier}
}

@article{conlon2015hedgehogs,
  title={Hedgehogs are not colour blind},
  author={Conlon, David and Fox, Jacob and R{\"o}dl, Vojt{\v{e}}ch},
  journal={Journal of Combinatorics},
  volume={8},
  pages={475--485},
  year={2017},
  publisher={Elsevier}
}

@article{fox2020ramsey,
  title={On {R}amsey numbers of hedgehogs},
  author={Fox, Jacob and Li, Ray},
  journal={Combinatorics, Probability and Computing},
  volume={29},
  number={1},
  pages={101--112},
  year={2020},
  publisher={Cambridge University Press}
}

@article{kostochka2006ramsey,
  title={On {R}amsey numbers of uniform hypergraphs with given maximum degree},
  author={Kostochka, Alexandr V and R{\"o}dl, Vojtech},
  journal={Journal of Combinatorial Theory, Series A},
  volume={113},
  number={7},
  pages={1555--1564},
  year={2006},
  publisher={Elsevier}
}

@inproceedings{dubroff2023tower,
  title={Tower gaps in multicolour {R}amsey numbers},
  author={Dubroff, Quentin and Gir{\~a}o, Ant{\'o}nio and Hurley, Eoin and Yap, Corrine},
  booktitle={Forum of Mathematics, Sigma},
  volume={11},
  pages={e84},
  year={2023},
  organization={Cambridge University Press}
}

@article{allen2026ramsey,
  title={{R}amsey numbers for 1-degenerate 3-graphs},
  author={Allen, Peter and Boyadzhiyska, Simona and Pavez-Sign{\'e}, Mat{\'\i}as},
  journal = {The Electronic Journal of Combinatorics},
  volume = {33},
  number = {2},
  pages = {P2.61},
  year = {2026},
}

@article{draganic2025size,
  title={Size-{R}amsey numbers of graphs with maximum degree three},
  author={Dragani{\'c}, Nemanja and Petrova, Kalina},
  journal={Journal of the London Mathematical Society},
  volume={111},
  number={3},
  pages={e70116},
  year={2025},
  publisher={Wiley Online Library}
}

@article{draganic2026size,
  title={Size-{R}amsey Numbers of Structurally Sparse Graphs},
  author={Dragani{\'c}, Nemanja and Kaufmann, Marc and Correia, David Munh{\'a} and Petrova, Kalina and Steiner, Raphael},
  journal={Random Structures \& Algorithms},
  volume={68},
  number={2},
  pages={e70059},
  year={2026},
  publisher={Wiley Online Library}
}

@article{kohayakawa2019size,
  title={The size {R}amsey number of short subdivisions of bounded degree graphs},
  author={Kohayakawa, Yoshiharu and Retter, Troy and R{\"o}dl, Vojt{\v{e}}ch},
  journal={Random Structures \& Algorithms},
  volume={54},
  number={2},
  pages={304--339},
  year={2019},
  publisher={Wiley Online Library}
}
 
\end{document}